%% file: quotients.tex
\documentclass[twoside]{amsart}

\input{abbreviations-cd.tex}

\usepackage[arrow,curve,matrix]{xy}

\newcounter{intro}

\newtheorem{intro-conjecture}[intro]{Conjecture}
\newtheorem{intro-corollary}[intro]{Corollary}
\newtheorem{intro-theorem}[intro]{Theorem}
\newtheorem{intro-problem}[intro]{Problem}

\DeclareMathOperator{\Alb}{Alb}

\newcommand{\R}{\mathbf R}
\newcommand{\D}{\mathbf D}
\newcommand{\EE}{\mathcal{E}}
\newcommand{\FF}{\mathcal{F}}

\def\D{\mathbf{D}}
\def\L{\mathbf{L}}
\def\RR{\mathbf{R}}
\def\X{\mathcal{X}}
\def\Y{\mathcal{Y}}
\def\Z{\mathcal{Z}}
\def\Pic{{\rm Pic}}
\def\Aut{{\rm Aut}}
\def\I{\mathcal{I}}
\def\Aff{{\rm Aff}}
\def\G{\mathcal{G}}
\def\H{\mathcal{H}}
\def\OO{{\mathcal O}}
\def\id{{\rm id}}
\def\E{\mathcal{E}}
\def\F{\mathcal{F}}
\def\P{\mathcal{P}}

\newcommand{\newpar}[1]{\subsection{\texorpdfstring{}{}}}
\newcommand{\parref}[1]{\hyperref[#1]{\S\ref*{#1}}}

\begin{document}

%========================================================
\title{Derived 
equivalences of smooth stacks and orbifold Hodge numbers}

\author[M.~Popa]{Mihnea Popa}
\address{Department of Mathematics, University of Illinois at Chicago,
851 S. Morgan Street, Chicago, IL 60607, USA } 
\email{{\tt mpopa@uic.edu}}
\thanks{The author was partially supported by the NSF grant DMS-1101323.}
%\date{\today}

\setlength{\parskip}{.04 in}

\dedicatory{Dedicated to Professor Kawamata on the occasion of his $60$th birthday, with great admiration.}

\maketitle
%========================================================

\section*{Introduction}

An important step in the development of the parallelism between derived equivalences and the minimal model program,  
as emphasized especially in the work of Kawamata (see \cite{kawamata4} for a survey), is to extend  results about smooth projective Fourier-Mukai partners to the singular case. While in general there are foundational issues still to be resolved, good progress has been made in the case of varieties with quotient singularities $X$, where the natural object to consider is the bounded derived category of coherent sheaves $\D (\X) := \D^b ({\rm Coh} (\X))$ on the associated canonical smooth Deligne-Mumford stack $\X$ \cite{BKR}, \cite{kawamata1}, \cite{kawamata3}. The main result of this paper is an addition in this direction, regarding the behavior of the orbifold cohomology and Picard variety.

\begin{intro-theorem}\label{general}
Let $X$ and $Y$ be normal projective varieties of dimension $n$, with quotient singularities, 
and let $\X$ and $\Y$ be the associated canonical smooth Deligne-Mumford 
stacks (or orbifolds). Assume that $\D(\X) \simeq \D(\Y)$. Then

\noindent 
(i) For every $- n \le i \le n$, one has an isomorphism
$$\underset{p-q= i}{\bigoplus} H^{p,q}_{\rm orb}(\X) \simeq \underset{p-q= i}{\bigoplus} H^{p,q}_{\rm orb} (\Y),$$
where $H^{p,q}_{\rm orb}(\X)$ are the orbifold Dolbeaut cohomology groups (see \S2.1).

\noindent
(ii)  $\Pic^0 (\X)$ and $\Pic^0 (\Y)$ are isogenous; in particular $h^{0,1}_{\rm orb} (\X) =  h^{0,1}_{\rm orb} (\Y)$.
\end{intro-theorem}

This is inspired by similar results in the smooth setting, where: (i) is a consequence of the homological Fourier-Mukai transform
and the Grothendieck-Riemann-Roch theorem \cite{orlov} Ch.2 (see also \cite{huybrechts} \S5.2), or of the invariance of Hochschild 
homology \cite{caldararu1} \S8, \cite{orlov} Ch.2; (ii) is the main result of \cite{PS}, relying on Rouquier's study of the connected 
component of the group of derived autoequivalences, and the study of actions of non-affine algebraic groups.
It makes progress towards the following:

\begin{intro-problem}\label{hodge_invar}
Let $X$ and $Y$ be normal projective varieties with quotient singularities, and let $\X$ and $\Y$ be the associated smooth 
Deligne-Mumford stacks. Assume that $\D(\X) \simeq \D(\Y)$. Given $p, q \in \QQ$, does one have 
$h^{p,q}_{\rm orb} (\X) =  h^{p,q}_{\rm orb} (\Y)$?
\end{intro-problem}

One reason this problem is more subtle than in the smooth projective case is that when the singularities of $X$ and $Y$ are not Gorenstein, orbifold Hodge numbers indexed by \emph{rational} numbers $p, q$, as opposed to integers, are guaranteed to enter the picture (see \S2.1). Thus even in low dimension Theorem \ref{general} potentially does not provide enough relations to solve for all orbifold Hodge numbers; in general we can only deduce  the following about individual ones:

\begin{intro-corollary}
Let $X$ and $Y$ be normal projective varieties with quotient singularities, and let $\X$ and $\Y$ be the associated smooth Deligne-Mumford stacks. Assume that $\D(\X) \simeq \D(\Y)$. Then
$$h^{n,0}_{\rm orb} (\X) =  h^{n,0}_{\rm orb} (\Y), \,\,\,\, h^{n-1,0}_{\rm orb} (\X) =  h^{n-1,0}_{\rm orb} (\Y), \,\,\,\,{\rm and 
}\,\,\,\,h^{1,0}_{\rm orb} (\X) =  h^{1,0}_{\rm orb} (\Y).$$
In particular, if $X$ and $Y$ are of dimension up to three, then 
$$h^{p,0}_{\rm orb} (\X) =  h^{p,0}_{\rm orb} (\Y), \,\,\,\,{\rm for~all~} p.$$
\end{intro-corollary}

The last identity follows directly from Theorem \ref{general}, while the first two are explained in Corollary \ref{high}.
However, as in the smooth case \cite{PS}, the theorem does suffice in the Gorenstein case in dimension up to three.

\begin{intro-corollary}\label{Gorenstein}
Assume in addition that $X$ and $Y$ are Gorenstein of dimension up to three. Then 
$$h^{p,q}_{\rm orb} (\X) =  h^{p,q}_{\rm orb} (\Y), \,\,\,\,{\rm for~all~} p, q.$$
\end{intro-corollary}
\begin{proof}
We will note below that in the Gorenstein case we only have $p,q \in \ZZ$, and that (in general) Serre duality 
$h^{p,q}_{\rm orb} (\X) = h^{n - p,n - q}_{\rm orb} (\X)$ still holds.
Therefore the statement follows immediately using all of the numerical information given by Theorem \ref{general}, 
as in \cite{PS} Corollary C.
\end{proof}

It is explained in \S5 that this last corollary can also be deduced using the more sophisticated derived and cohomological McKay correspondence, combined with \cite{PS}. However, in the present context it is a consequence of general methods that do not rely on the existence of crepant resolutions; it just happens to be a case of Theorem \ref{general} where one can fully solve the system of equations in the Hodge numbers that it provides. In \S5 I also speculate on analogues of Problem \ref{hodge_invar} for arbitrary singularities in the minimal model program,  with the place of orbifold Hodge numbers being taken by the stringy Hodge numbers.

The structure of the proof of Theorem \ref{general} completely follows the approach in the smooth projective case, so from the strategy point of view it is fair to say that there are no truly new ideas here. There are however 
various technical difficulties related to both definitions and proofs in the orbifold case, whose solution is facilitated by
recent developments in the study of Deligne-Mumford stacks. Along the way,  in \S2 I provide a cohomological orbifold Fourier-Mukai transform using an 
orbifold Mukai vector inspired by the Grothendieck-Riemann-Roch theorem for quotient stacks; this leads to Theorem \ref{general}(i). 
In \S3 I present a proof of an orbifold analogue of Rouquier's theorem on the invariance of $\Aut^0 \times \Pic^0$, which is 
applied in \S4 in order to deduce Theorem \ref{general}(ii) along the lines of \cite{PS}.  

\noindent
{\bf Acknowledgements.} 
This paper could not exist without Professor Yujiro Kawamata's systematic study of derived categories associated to singular varieties, and their relationship with the minimal model program. I have also benefitted greatly from ideas in the joint work with C. Schnell \cite{PS}. I would like to thank them, as well as A. C\u ald\u araru, I. Coskun, D. Edidin, M. Olsson, T. Yasuda and E. Zaslow for answering my questions and for providing numerous references.

\section{Cohomological Fourier-Mukai transform for orbifolds}

\subsection{Orbifold cohomology of Deligne-Mumford stacks}
In this paper, a Deligne-Mumford stack is more precisely a separated Deligne-Mumford stack of finite type over $\CC$, with a scheme as coarse moduli space. If $\X$ is one such, we can consider the abelian category ${\rm Coh}(\X)$ of coherent sheaves on $\X$, and its bounded derived category $\D (\X)$. 

We will be mostly concerned with orbifolds. These can be defined using the standard orbifold terminology (cf. e.g. \cite{BG} Ch.4 or \cite{kawamata1} \S2), or equivalently in the language of Deligne-Mumford stacks (see e.g. \cite{MP} or \cite{kawamata3} \S4); namely an \emph{orbifold} is a smooth Deligne-Mumford stack $\X$ whose generic point has trivial automorphism group. We will always assume that the orbifold has no pseudo-reflections, i.e. no codimension one fixed 
loci for the elements of the local groups; the coarse moduli scheme of $\X$ is then a variety $X$ with quotient singularities, and $\X$ is determined by $X$. Note that conversely every variety with quotient singularities arises in this fashion, i.e. is the coarse moduli scheme of an orbifold without pseudo-reflections (see e.g. \cite{yasuda1} Lemma 4.29). In what follows we consider only projective orbifolds.

%\begin{lemma}\label{birat_hodge}
%Let $\X$ be an orbifold, with coarse moduli variety $X$, and let $Y$ be any resolution of singularities of $X$. Then 
%$$h^{0,q}_{{\rm orb}} (\X) = h^{0,q} (Y) \,\,\,\,{\rm for~ all~} q.$$ 
%\end{lemma}

In the standard orbifold language, orbifold cohomology and orbifold Hodge numbers have been defined and studied in \cite{CR}. 
Here I follow the presentation of Yasuda \cite{yasuda1} \S3.4, \cite{yasuda2} \S4.3, who describes the algebraic theory in the context
of stacks. Given a smooth Deligne-Mumford stack $\X$ over $\CC$,  with projective coarse moduli scheme, one defines the 
\emph{orbifold cohomology groups} of $\X$ as
$$H^i_{{\rm orb}} (\X, \QQ) : = \bigoplus_{\Z \subset \I \X}
H^{i - 2a (\Z)} (Z, \QQ) \otimes \QQ \big(-a (\Z)\big), $$
where $\QQ \big(-a (\Z)\big)$  is a Tate twist.
Here $\I\X$ denotes the inertia stack of $\X$; when $\X$ is an orbifold, $\I\X$ is finite and 
unramified over $\X$, and as above it can be written 
as the disjoint union of closed stacks $\Z$ (which are themselves inertia stacks for various fixed
loci). We denote by $Z$ the coarse moduli space of each such component; since $\I \X$ is a smooth stack and the map
$Z \rightarrow X$ is quasi-finite, Z is again a variety and has quotient singularities. It is well known  that in this case the rational cohomology groups $H^k (Z ,\QQ)$ continue to carry pure Hodge structures, with Serre and Poincar\'e duality; 
see \cite{steenbrink} \S1.
The number $a(\Z)$ is the \emph{shift number} of the component $\Z$, defined as 
follows: assume that the generic point of $\Z$, which consists of a pair $(x, \alpha)$ such that $x$ is a closed point of $\X$ and $\alpha \in {\rm Aut} (x)$, is such that ${\rm ord}(\alpha) = l$. We have then that $\alpha$ acts on the tangent space $T_x \X$ suchthat in a suitable bases this action is given by  a diagonal matrix 
$${\rm diag} \big(\nu_l^{a_1}, \ldots, \nu_l^{a_n} \big),$$
where $\nu_l$ is an $l$-th root of unity, $1 \le a_i \le l$, and $n = {\rm dim}~\X$.
Then 
$$a(\Z) : = n - \frac{1}{n}\sum_{i=1}^n a_i.$$
It is a standard fact that this depends only on the component $\Z$.
The \emph{orbifold Hodge numbers} of $\X$ are the Hodge numbers associated to the Hodge structure 
$H^i_{{\rm orb}} (\X, \QQ)$. Note that this is a pure Hodge structure of weight i; see \cite{yasuda2} Lemma 76. We have  
\begin{equation}\label{orbifold_hodge_def}
h^{p,q}_{{\rm orb}}(\X) = \sum_{\Z \subset \I \X} h^{p - a(\Z), q- a(\Z)} (Z).
\end{equation}
Some remarks are in order: all $a(\Z)$ are integers precisely when $X$ has Gorenstein quotient singularities. This is well known to correspond to the case when the matrix ${\rm diag} \big(\nu_l^{a_1}, \ldots, \nu_l^{a_n} \big)$ is in ${\rm SL}_n (\CC)$. In general however, in the above we are using the following conventions: 

\noindent
$\bullet$\,\, $H^i (Z, \QQ) = 0 \,\,\,\, {\rm if}\, i \notin \ZZ$.

\noindent
$\bullet$\,\,  $H^{p,q} (Z) = 0 \,\,\,\, {\rm if}\, p,q \notin \ZZ$.

\noindent
Thus in (\ref{orbifold_hodge_def}) the sum is taken over all $p, q \in \QQ$ such that $p - a (\Z), q - a(\Z) \in \ZZ$.

The orbifold Hodge numbers are easily identified in the special case of birationally invariant Hodge numbers.

\begin{lemma}\label{birat_hodge}
Let $X$ be a normal projective variety with quotient singularities, $\X$ the associated orbifold, and $\widetilde X$ 
a resolution of singularities of $X$.
Then
$$h^{0,q}_{{\rm orb}}(\X) = h^{0,q} (\widetilde X)\,\,\,\,{\rm for~all~} q.$$
\end{lemma}
\begin{proof}
It is clear from (\ref{orbifold_hodge_def}) that the only contribution to $h^{0,q}_{{\rm orb}}(\X)$ can come from 
the ``untwisted sector", i.e. when $Z = X$, corresponding to the case $a (\Z) = 0$. But in the pure Hodge structure
on $H^q (X, \QQ)$ described in \cite{steenbrink}, the $H^{q, 0}$ part can be computed on any resolution of singularities.
\end{proof}

\subsection{Cohomological Fourier-Mukai transform}
According to Kawamata \cite{kawamata3} Theorem 1.1, Orlov's theorem continues to hold in this setting: if $\X$ and $\Y$ are the orbifolds associated to normal projective varieties with quotient singularities, and $\Phi$ is an equivalence between the triangulated categories $\D(\X)$ and $\D(\Y)$, then there exists a (unique up to isomorphism) object  $\E \in  \D(\X\times \Y)$ such that 
$$\Phi = \Phi_{\E}: \D (\X) \rightarrow \D(\Y),\,\,\,\, \Phi_{\E} (\cdot) = \RR{p_2}_* \big(p_1^*(\cdot)\overset{\L}{\otimes} 
\E\big),$$ 
the Fourier-Mukai transform induced by $\E$. We will see that, as in the setting of smooth projective varieties, this induces a cohomological Fourier-Mukai transform, this time at the level of orbifold cohomology.

\noindent
{\bf Assumption.} The first part of the discussion below works for arbitrary quotient stacks. As explained in \cite{edidin} \S3.2, this is a very mild restriction;  for instance, every Deligne-Mumford stack that has the resolution property is a quotient stack. This is certainly the case for orbifolds, according to \cite{totaro} Theorem 1.2 (see also \cite{kawamata3} Theorem 4.2).

We denote by $K_0 (\X)$ the Grothendieck group of vector bundles on $\X$, or equivalently of coherent sheaves as $\X$ has the resolution property.  As usual, if $\X$ is smooth, there exist a Chern character and a Todd class
$${\rm ch, ~Td}: K_0 (\X) \longrightarrow {\rm Ch}^* (\X) \otimes \QQ.$$
Note that over $\QQ$ we have isomorphisms 
$${\rm Ch}^* (\X) \otimes \QQ \simeq {\rm Ch}^* (X) \otimes \QQ \,\,\,\,
{\rm and} \,\,\,\, H^* (\X, \QQ) \simeq H^* (X, \QQ),$$
where $X$ is the coarse moduli space of $\X$. In particular, via the cycle class map of $X$, we
can also consider the Chern character at the cohomology level:
$${\rm ch}: K_0 (\X) \longrightarrow H^* (\X, \QQ).$$
A detailed discussion of these constructions for Deligne-Mumford stacks can be found for instance 
in \cite{edidin} and the references therein.

Following \cite{toen} and \cite{edidin}, the Grothendieck-Riemann-Roch theorem for (quotient) Deligne-Mumford 
stacks is best expressed in terms of a Chern character which takes values in the complex 
cohomology of $\I\X$, i.e. in the orbifold cohomology of $\X$. More precisely, inspired by 
the localization theorem one produces a homomorphism
$$I{\tau_{\X}}: K_0(\X) \longrightarrow H^* ( \I \X, \CC),$$
which as explained in \cite{edidin} \S5 induces an isomorphism $K_0 (\X)_{\CC} \rightarrow {\rm Ch}^*(\I\X)\otimes \CC$, 
and which is given by
$$I_{\tau_{\X}} ([V]) =  {\rm ch} (p^* V) \cdot  {\rm ch} (\alpha_{\X}^{-1}) \cdot {\rm Td}(\I\X).$$
Here $p: \I\X \rightarrow \X$ is the natural projection morphism, while $\alpha_{\X}^{-1}$ is a ``localization" class in 
$K_0 (\I \X)$; see \cite{toen} p.29.
The Grothendieck-Riemann-Roch theorem in \cite{toen} \S4 and \cite{edidin} \S5 
states that if $f: \X \rightarrow \Y$  is a proper 
morphism of quotient stacks with the resolution property, then there is a commutative diagram
$$\xymatrix{
K_0 (\X)\ar[r]^{ \I_{\tau_{\X}} \hspace{4mm}} \ar[d]_{f_*}  & {\rm Ch}^*(\I\X)\otimes \CC \ar[d]^{f_*} \\
K_0 (\Y)  \ar[r]^{\I_{\tau_{\Y}} \hspace{4mm}} & {\rm Ch}^*(\I\Y)\otimes \CC.}
$$
As a final piece of notation, following the standard case, for every stack $\X$ as above and every $\F \in \D(\X)$ we define 
the \emph{orbifold Mukai vector} associated to $\F$ as
$$v_{{\rm orb}}(\F) : = {\rm ch} (p^* \F) \cdot \sqrt{ {\rm ch} (\alpha_{\X}^{-1}) \cdot {\rm Td}(\I\X)} \in H^* (\I \X, \CC).$$

Going back to orbifolds, let $\X$ and $\Y$ be two such, and 
$\E \in \D(\X \times \Y)$ inducing the integral functor $\Phi_{\E} : \D(\X) \rightarrow \D(\Y)$.
%Following the standard case, we consider  the associated Mukai vector 
%$$v(\E) : = {\rm ch} (\E) \cdot \sqrt{{\rm td} (\X\times \Y)} \in H^* (\X\times \Y, \QQ),$$
%and the corresponding cohomological Fourier-Mukai transform 
%$$\Phi_{\E}^H : H^* (\X, \QQ) \longrightarrow H^*(\Y, \QQ), \,\,\,\,\alpha \mapsto {p_2}_* \big(p_1^* \alpha \cdot v (\E)\big)$$
%Denote by $p_{\X} : \I\X\rightarrow \X$ and $p_{\Y} : \I\Y\rightarrow \Y$ the natural projections, and consider
%the object $\P_{\E} : = (p_{\X} \times p_{\Y})^* \E$ on $\I\X \times \I\Y$. (Do we need it?) 
Using $v_{{\rm orb}}(\E)  \in H^* (\I \X \times \I \Y, \CC)$, we define the \emph{orbifold cohomological Fourier-Mukai transform} as
$$\Phi_{\E}^H : H^* (\I\X, \CC) \longrightarrow H^*(\I\Y, \CC), \,\,\,\,\alpha \mapsto {p_2}_* \big(p_1^* \alpha \cdot v_{{\rm orb}} (\E)\big).$$
Note that the product above is the usual cup product rather than the 
orbifold cup-product of \cite{CR}.

The following results are analogues of standard facts in the case of smooth projective varieties; see for instance 
\cite{huybrechts} Corollary 5.29, Proposition 5.33 and Proposition 5.39.

\begin{proposition}\label{general_isomorphism}
If $\Phi_{\E}$ is an equivalence of derived categories, then the induced orbifold cohomological Fourier-Mukai transform 
$$\Phi_{\E}^H : H^* (\I\X, \CC) \longrightarrow  H^* (\I\Y, \CC)$$
is a (non-graded) isomorphism of vector spaces.
\end{proposition}
\begin{proof}
Denote by $\F$ the right adjoint of $\E$. The composition 
$\Psi_{\F} \circ \Phi_{\E}$ is the identity on $\D(\X)$, given by the kernel $\OO_{\Delta_{\X}}$. 
Similarly, the composition 
$\Phi_{\E} \circ \Psi_{\F}$ is the identity on $\D(\Y)$, given by the kernel $\OO_{\Delta_{\Y}}$. 
It is not hard to see then that the composition $\Psi_{\F}^H \circ \Phi_{\E}^H$ is the endomorphism 
$\Phi_{\OO_{\Delta_{\X}}}^H$ of $H^* (\I\X , \CC)$, and similarly in the opposite direction.

It suffices to show that this endomorphism is the identity. Note that it corresponds to the Mukai vector
$$v_{{\rm orb}} (\OO_{\Delta_{\X}})  =  \frac{I_{\tau_{\X \times \X}}}{\sqrt{ {\rm ch} (\alpha_{\X \times \X}^{-1}) \cdot 
{\rm Td}(\I\X \times \I\X)}} \in H^* (\I \X \times \I \X, \CC).$$
In order to prove the assertion, we apply the Grothendieck-Riemann-Roch theorem stated above to the 
diagonal embedding $i_{\X} : \X \hookrightarrow \X \times \X$. This induces a commutative diagram 
$$\xymatrix{
\I\X \ar[r]^{ i_{\I\X} \hspace{4mm}} \ar[d]_{p_{\X}}  & \I\X \times \I\X   \ar[d]^{p_{\X}\times p_{\X}} \\
\X \ar[r]^{ i_{\X} \hspace{4mm}} & \X \times \X}
$$
and the theorem applied to the object $\OO_{\Delta_{\X}}$ gives
$$I_{\tau_{\X\times \X}} (\OO_{\Delta_{\X}}) = (i_{\I\X})_* I_{\tau_{\X}} (\OO_{\X}). $$
Note now that 
$$I_{\tau_{\X}} (\OO_{\X}) = {\rm ch} (\alpha_{\X}^{-1}) \cdot {\rm td}(\I \X) = i_{\I\X}^* \sqrt{ {\rm ch} (\alpha_{\X \times \X}^{-1}) \cdot 
{\rm Td}(\I\X \times \I\X)},$$
where the first equality follows from the definition, while the second is just restriction to the diagonal.
The projection formula then implies
$$v_{{\rm orb}} (\OO_{\Delta_{\X}})  = {i_{\I \X}}_* (1),$$
from which it is immediate to conclude that $\Phi_{\OO_{\Delta_{\X}}}^H$ is the identity.
\end{proof}

\begin{proposition}\label{hodge_inclusion}
Let $\Phi_{\E} : \D(\X) \rightarrow \D(\Y)$ be an equivalence of derived categories of orbifolds.
Then, for every $p$ and $q$, we have
$$\Phi_{\E}^H \big( H^{p,q}_{\rm orb} (\X) \big) \subset \bigoplus_{r-s = p-q} H^{r,s}_{\rm orb} (\Y).$$
\end{proposition}
\begin{proof}
The orbifold Mukai vector $v_{{\rm orb}} (\E) \in H^* (\I\X \times \I\Y , \CC)$ is an algebraic class, so
$$v_{{\rm orb}} (\E) = \sum \alpha^{u, v} \boxtimes \beta^{r,s}, \,\,\,\, r-s = v - u,$$
with $\alpha^{u, v} \in H^{u,v}_{\rm orb} (\X)$ and $\beta^{r, s} \in H^{r,s}_{\rm orb} (\Y)$.

Let now $\alpha^{p, q}$ be a class in $H^{p,q}_{\rm orb} (\X)$. In order for $(\alpha^{p,q} \cdot \alpha^{u, v}) 
\boxtimes \beta^{r,s}$ to contribute non-trivially to $q_* (p^* \alpha^{p,q} \cdot v(\E))$, we must have that 
$\alpha^{p,q} \cdot \alpha^{u, v} \in H^{n, n}_{\rm orb} (\X)$. This implies that $p + u= q + v = n$, and in 
particular
$$p - q = v - u = r-s.$$
\end{proof}

Putting together Proposition \ref{general_isomorphism} and Proposition \ref{hodge_inclusion}, we obtain Theorem
\ref{general}(i).

\begin{corollary}\label{columns}
If $\Phi_{\E} : \D(\X) \rightarrow \D(\Y)$ is an equivalence of derived categories of orbifolds, then $\Phi_{\E}^H$ induces
an isomorphism
$$\bigoplus_{p-q = i} H^{p,q}_{\rm orb} (\X) \simeq \bigoplus_{p-q = i} H^{p,q}_{\rm orb} (\Y)$$
for every $i = - n, \ldots, n$ and $p, q \in \QQ$.
\end{corollary}

For the sake of completeness, note finally that the Riemann-Roch formula stated above also implies the analogue of the usual 
commutation of the derived and cohomological Fourier-Mukai transforms, again with a very similar proof which I do not repeat here.

\begin{proposition}
Let $\E  \in \D (\X \times \Y)$. Then for every $A \in \D(\X)$
$$\Phi^H_{\E} \big(v_{{\rm orb}} (A)\big) \simeq  v_{{\rm orb}} \big(\Phi_{\E} (A)\big).$$ 
\end{proposition}

\subsection{Further remarks in the case of global quotients.}
When our orbifold is a global quotient $\X = [Z/G]$, then just as in the case of smooth projective varieties
the result in Corollary \ref{columns} can be interpreted 
as the derived invariance of orbifold Hochschild homology, after applying the Hochschild-Kostant-Rosenberg 
isomorphism on all (covers of the) components of the inertia stack. As A. C\u ald\u araru points out, this approach should work for all orbifolds, but the corresponding interpretation of Hochschild homology (i.e. the analogue of Theorem \ref{hochschild_decomposition} below) has not yet been proved in the non-global case.

I start by recalling a slightly different interpretation of the orbifold cohomology groups in the case of global quotients. 
We assume henceforth that $Z$ is a smooth projective complex variety of dimension $n$, 
and $G$ a finite group of automorphisms acting on $Z$. Given an element $g \in G$, we denote by:

\noindent
$\bullet$~ $Z^g$ the fixed point set of $g$, and $Z^g_\alpha$ its 
connected components.

%\noindent
%$\bullet$~$C(g)$ its  centralizer, and $[g]$ its conjugacy class in $G$.

\noindent
$\bullet$~$a(g, x)$ the \emph{age} of $g$ at a point $x \in Z^g$; this is defined as in \S2.1, and depends only 
on the connected component $Z^g_{\alpha}$ of $x \in Z^g$, so it will be denoted $a(g, \alpha)$.

An equivalent interpretation of orbifold singular cohomology and Dolbeaut cohomology spaces 
was given by Fantechi-G\"ottsche in \cite{FG} \S1. Concretely, for the complex orbifold cohomology groups of $[Z/G]$ one has
$$H^i_{{\rm orb}} ([Z/G], \CC) \simeq \left( \bigoplus_{g, \alpha} H^{i - 2 a(g, \alpha)} 
(Z^g_\alpha, \CC)\right)^G,$$
where the invariants are taken with respect to the action of $G$ by conjugation.
Similarly, the orbifold Dolbeaut cohomology groups of $[Z/G]$ satisfy 
\begin{equation}\label{hodge_global}
H^{p,q}_{{\rm orb}} ([Z/G])\simeq 
\left(\bigoplus_{g, \alpha} H^{p - a(g, \alpha), q - a(g, \alpha)} (Z^g_\alpha)\right)^G,
\end{equation}
with dimensions equal to the orbifold 
Hodge numbers $h^{p,q}_{{\rm orb}} ([Z/G])$.

In the context of this section, by \emph{orbifold Hochschild homology} we mean the Hochschild homology 
${\rm HH}_{\bullet} ([Z/G])$ of the stack.
By analogy with the isomorphisms above one has the following result, 
which appears in various forms in \cite{baranovsky} Theorem 1, \cite{caldararu2} Theorem 6.16, and \cite{ganter} Theorem 6.3 (the cohomology version). 

\begin{theorem}\label{hochschild_decomposition}
There is an isomorphism
$${\rm HH}_{\bullet} ([Z/G]) \simeq \left( \bigoplus_{g, \alpha} {\rm HH}_{\bullet} (Z^g_\alpha)\right)^G,$$
where the spaces on the right hand side are the direct sums of the usual Hochschild homology spaces of $X^g_\alpha$, and the invariants are taken with respect to the action by conjugation.
\end{theorem}

Note that, denoting by $i$ the diagonal embedding of $Y$, I am using the indexing notation from the papers above, namely
$${\rm HH}_i (Y) : = {\rm Ext}^{i- \dim Y}_{Y\times Y} ( i _* \OO_Y, i_* \omega_Y),$$
which is slightly different from that in \cite{huybrechts} \S6.
With this convention, if $Y$ is smooth projective and $n = \dim Y$, the Hochschild-Kostant-Rosenberg
isomorphism (see e.g. \cite{huybrechts} \S6.1) reads 
\begin{equation}\label{hkr}
{\rm HH}_i (Y) \simeq \bigoplus_{p- q = i} H^{p, q} (Y).
\end{equation}

Combining this with Theorem \ref{hochschild_decomposition}, we have the following orbifold analogue:

\begin{proposition}\label{orbifold_HKR}
If $n = \dim Z$, for each $i$ there is an isomorphism
$${\rm HH}_i ([Z/G]) \simeq \bigoplus_{p-q = i} H^{p,q}_{{\rm orb}} ([Z/G]).$$
\end{proposition}
\begin{proof}
Theorem \ref{hochschild_decomposition} initiates the following sequence of isomorphisms:

\begin{eqnarray*}
{\rm HH}_i ([Z/G]) & \simeq & \left( \bigoplus_{g, \alpha} {\rm HH}_i (Z^g_\alpha)\right)^G   \\
& \simeq & \left( \bigoplus_{g, \alpha} \bigoplus_{p'-q' = i} H^{p', q'} (Z^g_\alpha) \right)^G \\
&\simeq  &   \left( \bigoplus_{g, \alpha} \bigoplus_{p-q = i} H^{p - a(g, \alpha), q - a(g, \alpha)} 
(Z^g_\alpha) \right)^G \\
&\simeq  &   \bigoplus_{p-q = i} \left(\bigoplus_{g, \alpha} H^{p - a(g, \alpha), q - a(g, \alpha)} 
 (Z^g_\alpha) \right)^G \\
 &\simeq &  \bigoplus_{p-q = i} H^{p,q}_{{\rm orb}} ([Z/G]),
\end{eqnarray*}
where the last three sums are taken over $p$ and $q$ rational such that $p - a(g, \alpha), q - a(g, \alpha) \in \ZZ$.
The second isomorphism uses (\ref{hkr}), while the last is (\ref{hodge_global}).
\end{proof}

Since Hochschild homology is well known to be an invariant of the derived category, this gives in particular 
another approach to Corollary 2.4 in the global quotient case.

\section{Rouquier's theorem in the orbifold setting}

\subsection{Orbifold Picard group}
The group $\Pic (\X)$ parametrizes isomorphism classes of line bundles on the Deligne-Mumford stack $\X$, with 
the tensor product operation. (In the case of orbifolds, this is equivalent to the notion of orbifold line bundles on the 
coarse space $X$, as for instance in \cite{BG} \S4.4.) We denote by $\Pic^0 (\X)$ its connected component of the identity, 
which is an abelian variety determined via the standard:

\begin{lemma}\label{orbifold_picard}
Let $\X$ be a normal projective variety with quotient singularities, $\X$ the associated orbifold, and $\widetilde X$ 
a resolution of singularities of $X$. Then 
$$ \Pic^0 (\X) \simeq \Pic^0 (X) \simeq \Pic^0 (\widetilde X).$$
In particular, $\dim \Pic^0 (\X) = h^{0,1}_{{\rm orb}} (\X)$.
\end{lemma}
\begin{proof}
The numerical equality follows from the first statement and Lemma \ref{birat_hodge}. The second isomorphism in the first statement is essentially definitional (the Picard variety is a birational invariant, so independent of the smooth model chosen). To prove the first isomorphism, observe first that there is a standard inclusion
\begin{equation}\label{pic1}
\Pic (X) \hookrightarrow \Pic (\X), \,\,\,\, L \mapsto \pi^*L.
\end{equation}
%(These are called the \emph{absolute line orbibundles} in \cite{BG} Ch.4.)  
On the other hand, since 
$\X$ is projective, it has finite \emph{order} $r$; this is the least common multiple of the orders of the stabilizer groups 
of points on local covers of $X$ of the form $U \rightarrow U/G$. We get a group homomorphism 
\begin{equation}\label{pic2}
\Pic (\X) \longrightarrow \Pic (X), \,\,\,\, L \mapsto L^{\otimes r},
\end{equation}  
as taking the $r$-th power turns any local linearization of $L$ into the trivial linearization. 
%corresponding to an absolute orbibundle; see \cite{BG} Proposition 4.4.22. 
This is a finite surjective map, so combining (\ref{pic1}) and 
(\ref{pic2}) we immediately obtain the conclusion.
\end{proof}

As a well-known concrete example, recall that when $\X = [Z/G]$ be a global quotient orbifold, with $Z$ a smooth projective variety and $G$ a finite group of automorphisms of $Z$,  then
$$\Pic (\X) \simeq \Pic (Z; G),$$
the Picard group of $G$-linearized line bundles on $Z$. On the other hand, the group $\Pic(Z)^G$ of 
$G$-equivariant line bundles on $Z$ gives the Picard group of the coarse moduli scheme $Z/G$, 
and we have an exact sequence
$$0 \rightarrow H^1 (G, \CC^*) \rightarrow \Pic (Z; G) \rightarrow \Pic(Z)^G \rightarrow 
H^2 (G, \CC^*),$$
where the map in the middle is given by forgetting the linearization; see e.g. \cite{dolgachev} Remark 7.2. 
Note that this has a section induced by the inclusion 
$\Pic(Z)^G \hookrightarrow \Pic (Z; G)$ given by pullback. We conclude the 
isomorphism of connected components of the identity
$$\Pic^0  (\X) \simeq \Pic^0 (Z; G) \simeq \left( \Pic(Z)^G \right)^0.$$

\subsection{Orbifold automorphism group}
The automorphism group $\Aut (\X)$ has been studied extensively in the context of complex orbifolds, see e.g. \cite{fujiki}, \cite{nakagawa}; note that it has the structure of an algebraic group. In the general context of stacks, it is a more complicated categorical group object. However, when the 
Deligne-Mumford stack $\X$ is the canonical orbifold (without pseudo-reflections) associated to a variety $X$ with quotient singularities, $\Aut(\X)$ is a group, and in fact  (see e.g. \cite{FMN} Corollary 4.8):
$$ \Aut (\X) \simeq \Aut (X).$$

We denote by  $\Aut^0 (\X)$ its connected component of the identity; its tangent space at the origin is isomorphic to 
$H^0 (\X , T_{\X})$, the space of global vector fields on the stack $\X$.

%The problem of which automorphisms of $X$ can be lifted to orbifold automorphisms of $\X$ is quite subtle. For instance, in \cite{LMP} Theorem 5.1.1. it is essentially shown that in the case of a global quotient $X/G$ with $X$ simply connected (and with no codimension $1$ stabilizers), every automorphism of  $X/G$ can be lifted to one 
%f the orbifold $[X/G]$. In general, the fundamental group of $X$ imposes obstructions to lifting.

Chevalley's theorem on the structure of connected algebraic groups (see e.g. \cite{brion} p.1) implies that there is a natural exact sequence
$$0 \longrightarrow \Aff (\Aut^0 (\X)) \longrightarrow \Aut^0 (\X) \longrightarrow \Alb (\Aut^0 (\X)) \rightarrow 0,$$
where $\Aff (\Aut^0 (X))$ is the maximal connected affine subgroup of $\Aut^0 (\X)$, and $ \Alb (\Aut^0 (\X))$ is 
its Albanese variety. 
According to a theorem of Fujiki \cite{fujiki}, analogous to the Matsumura-Nishi theorem \cite{matsumura} in the case of projective varieties, the induced morphism of abelian varieties
\begin{equation}\label{matsumura}
\Alb (\Aut^0 (\X)) \rightarrow \Alb (\X)
\end{equation} 
has finite kernel. 

\begin{remark}
We have, for instance by the discussion of $\Pic^0 (\X)$ above, that 
$$\Alb (\X) \simeq \Alb (X) = \Alb (Y),$$
where $Y$ is any resolution of singularities of $X$. What is proved in \cite{fujiki}, and can also be derived by slightly 
expanding the context in \cite{matsumura}, is that the induced morphism
$$\Alb (\Aut^0 (X)) \longrightarrow \Alb (X)$$
has finite kernel. 
\end{remark}

Finally, note that there is a natural action of $\Aut (\X)$ 
on $\Pic^0 (\X)$ (and more generally on $\Pic (\X)$). 
Since $\Pic^0 (\X)$ is an abelian variety, it is not hard to check as in \cite{PS} \S5  that the 
restriction of this action to $\Aut^0 (\X)$ is trivial.

\subsection{Rouquier's theorem}
Our next goal is to extend Rouquier's result \cite{rouquier} Th\'eor\`eme 4.18
on the connected component of the group of derived autoequivalences, together with the 
related formula in \cite{PS} Lemma 3.1, to the orbifold  setting:

\begin{theorem}\label{rouquier}
Let $\X$ and $\Y$ be projective orbifolds such that there exists a Fourier-Mukai equivalence 
$\Phi_\E: \D(\X) \rightarrow \D(\Y)$, corresponding to an object $\E \in \D(\X \times \Y)$. Then there
exists an isomorphism of algebraic groups
$$F_\E: \Aut^0 (\X) \times \Pic^0 (\X) \longrightarrow \Aut^0 (\Y) \times \Pic^0 (\Y)$$ 
determined by the fact that $F_\E (\varphi, L) = (\psi, M)$ if and only if on $\X \times \Y$ one has
\begin{equation}\label{rouquier_formula}
	p_1^{\ast} L \otimes (\varphi\times \id)^{\ast} \EE \simeq
	p_2^{\ast} M \otimes (\id \times \psi)_{\ast} \EE. 
\end{equation}
\end{theorem}

\begin{remark}
A related statement, as in \cite{rosay}, is that the group of derived autoequivalences
$\Aut (\D(\X))$ is a locally algebraic group scheme whose identity component is given by
$$\Aut^0 (\D(\X)) \simeq \Aut^0 (\X) \times \Pic^0 (\X).$$
This provides a very natural interpretation of the derived invariance of this quantity. The precise 
formula (\ref{rouquier_formula}) is however important in what follows.
Note also that at the end of the day, due to the nature of the stacks we are considering, 
the Theorem really proves the invariance of the familiar quantity 
$$\Aut^0 (X) \times \Pic^0 (X)$$
attached to the singular variety $X$.
\end{remark}

\begin{proof}
This closely follows Rouquier's strategy in the case of smooth projective varieties, with extra care to check that all intermediate 
assertions continue to hold in the case of orbifolds.
Note first that $\Phi_\E$ induces a group isomorphism
\begin{equation}\label{auto_isom}
H : \Aut (\D(\X)) \longrightarrow \Aut (\D(\Y)), \,\,\,\,\Phi_\R \to \Phi_\E \circ \Phi_\R \circ \Phi_\E^{-1}.
\end{equation}
Here $ \Phi_\E^{-1}: \D(\Y) \rightarrow \D(\X)$ can be identified with the left adjoint of $\Phi_\E$, and is 
given by the kernel $\E_L := \E\otimes p^* \omega_{\Y} [n]$, where $\omega_{\Y}$ is the orbifold canonical
bundle of $\Y$, and $n = \dim \X = \dim \Y$. (Grothendieck duality and the 
Serre functor function as in the smooth projective case;  see \cite{kawamata1} \S2 and \cite{kawamata3} \S7.)
The equivalence $\Phi_\E \circ \Phi_\R \circ \Phi_\E^{-1}$ is given in turn by the convolution kernel 
$\E * \R * \E_L \in \D (\Y \times \Y)$.\footnote{Recall that given spaces $\X, \Y , \Z$ and objects $\F \in \D(\X \times \Y)$ 
and $\G \in \D (\Y \times \Z)$ respectively, the convolution of $\F$ and $\G$ is defined as 
$\F * \G : = \RR {p_{13}}_* (p_{12}^* \F \otimes p_{23}^* \G)$.} 

Recall now that a pair $(\varphi, L) \in G_{\X}: = \Aut^0 (\X) \times \Pic^0 (\X)$ 
induces an autoequivalence of $\D(\X)$ given by the kernel $(1, \varphi)_* L$, i.e. the line bundle $L$ supported on the 
graph $\Gamma_\varphi \subset \X \times \X$. 
We see the inclusion $ G_{\X} \hookrightarrow \Aut (\D(\X))$ as a family $\F$ of kernels in $\D(\X \times \X)$ 
parametrized by $G_{\X}$. Concretely, this is obtained as follows.  We consider the embedding 
$$f : \X \times G_{\X} \rightarrow \X \times \X \times G_{\X}$$
given by $(x, \varphi, L) \mapsto (x, \varphi(x), \varphi, L)$, and define
$$\F : = f_* p_{13}^* \P \in \D \left(\X \times \X \times G_{\X} \right),$$
where $\P$ is a Poincar\'e line bundle on $\X \times \Pic^0 (\X)$. Via the correspondence above, this maps to the family of 
kernels of autoequivalences of $\D(\Y)$
$$\H := p_{12}^* \E * \F * p_{12}^* \E_L \in \D \left(\Y \times \Y \times G_{\X} \right),$$
where $p_{12}$ denotes the projection of $\X \times \Y \times G_{\X}$ onto the first two factors.
Over the origin, the fiber of $\H$ is given by 
$$\H_0 = \E * \E_L \simeq \OO_{\Delta_{\Y}}.$$
 (It is well known that this holds for any Fourier-Mukai equivalence, and continues to hold in the orbifold case by the remark above.)

The claim now is that there is an open neighborhood of the origin $U \subset G_{\X}$ such that the restriction $\H_u \in \D(\Y \times \Y)$ 
over $u \in U$ continues to be a line bundle in $\Pic^0 (\Y)$ supported on the graph of an automorphism of $\Y$. 
Indeed, we can represent the object $\H$ by a bounded complex of orbifold vector bundles, and since $\H_0$ is the structure 
sheaf of the diagonal, by general semicontinuity of rank
%But since the base change theorem continues to hold in the orbifold setting (and much more generally for tame algebraic stacks; 
%see e.g. \cite{nironi} Theorem 1.7), 
it follows that there exists such a neighborhood so that this complex has cohomology only in degree zero. Moreover, by possibly 
shrinking $U$, we can also assume that for every $y \in \Y$ and $u \in U$ one has that $\H_{|\{y\} \times \Y \times \{u\}}$ is the 
structure sheaf of a point.
%and the same holds for all  $\H_{|\Y \times \{y\}  \times \{u\}}$. 
To conclude the claim, we can then apply Lemma \ref{functor} below.

We deduce that $H (U)$ is mapped into the algebraic subgroup $\Aut^0 (\Y) \times \Pic^0 (\Y) \subset \Aut (\D(\Y))$, and given the 
fact that $U$ generates $\Aut^0 (\X) \times \Pic^0 (\X)$ as a group, by restricting $H$ we get an induced morphism of algebraic groups
$$F_\E: \Aut^0 (\X) \times \Pic^0 (\X) \longrightarrow \Aut^0 (\Y) \times \Pic^0 (\Y).$$ 
Going in the other direction, the similar morphism obtained from $\Phi_\E^{-1}$ provides an inverse for $F_\E$.
Finally, given that the isomorphism is induced by $H$, the formula in (\ref{rouquier_formula}) follows exactly as in \cite{PS} Lemma 3.1.
\end{proof}

\begin{lemma}\label{functor}
Fix a projective orbifold $\X$, and consider the functor $G$ associating to every scheme of finite type $S$ the set $G(S)$ consisting of 
isomorphism classes of objects $\F$ in $\D(\X \times \X \times S)$ satisfying the following two conditions:

\noindent
$\bullet$\,\, for every $s \in S$, the restriction $\F_s \in \D(\X \times \X)$ of $\F$ to $\X \times \X \times \{s\}$ 
induces an autoequivalence of $\D(\X)$.

\noindent
$\bullet$\,\, for every $s \in S$ and every $x \in \X$, there exists $y \in \X$ such that ${\F_s}_{|\{x\} \times \X} \simeq \OO_y$. 

We denote by $\widetilde G$ the functor obtained by modding out $G(S)$ by the action of $\Pic (S)$ obtained by tensoring with 
pullbacks of line bundles on $S$. Then $\widetilde G$ is represented by the algebraic group scheme $\Aut (\X) \ltimes \Pic (\X)$.
\end{lemma} 
\begin{proof}
Fix $s \in S$. Since for every $x \in \X$ we have that $\L i_x^* \F_s$ is a sheaf, where $i_x$ is the inclusion
of $\{x\} \times \X$ in $\X \times \X$, by analogy with 
\cite{huybrechts} Lemma 3.31 one sees that $\F_s$ is an (orbifold) sheaf, flat over $\X$ via the first projection.
A similar reasoning then shows that $\F$ itself is a sheaf (flat over $S$).
For any $x \in \X$,  writing $f_s (x) = y$ as in the hypothesis defines a mapping $f_s : \X \rightarrow \X$. As the 
restriction of $\F$ to any point $(x,y)$ on the graph of this map is $\OO_{(x,y)}$, using the local sections of the 
orbifold sheaf $\F$ around every such point we obtain that $f_s$ is in fact 
 an orbifold morphism. 
 
Now completely analogously to the proof of \cite{huybrechts} Corollary 5.3 in the case of smooth projective varieties, 
given that $\F_s$ induces an autoequivalence, one concludes that $f_s$ is an automorphism and $\F_s$ is isomorphic 
to a line bundle supported on the graph of $f_s$ 
(which can  be thought of as a line bundle on $\X$).
 As the $\F_s$ form a flat family over $S$, we obtain an induced morphism 
 $f: S \rightarrow \Aut (\X)$, and it follows that $\F = \Psi_* \mathcal{L}$, where 
$$\Psi: \X \times S \longrightarrow \X \times \X \times S, \,\,\,\,(x, s) \mapsto(x, f_s(x), s),$$
 and with $\mathcal{L}$ an orbifold line bundle on $\X \times S$, rigidified after modding out by 
 pullbacks from $S$.
This produces a canonical element in ${\rm Hom} (S, \Aut (\X) \ltimes \Pic (\X))$.
\end{proof}

\section{The Picard variety}

The following is the orbifold analogue of the main result of \cite{PS}.

\begin{theorem}\label{isogeny}
Let $X$ and $Y$ be normal projective varieties with quotient singularities, and $\X$ and $\Y$ the associated orbifolds. 
If $\D (\X) \simeq \D(\Y)$, then ${\rm Pic}^0 (\X)$ and ${\rm Pic}^0 (\Y)$ are isogenous.
In particular $h^{0,1}_{{\rm orb}} (\X) = h^{0,1}_{{\rm orb}} (\Y)$ and $h^0 (\X , T_{\X}) = h^0 (\Y , T_{\Y})$.
\end{theorem}
\begin{proof}
We start in fact with the proof of the numerical equalities. The second follows from the first and Theorem \ref{rouquier}.
We show below that $h^{0,1}_{{\rm orb}} (\X)  \le h^{0,1}_{{\rm orb}} (\Y)$, the other inequality following by 
symmetry.

Let $\EE \in \D (\X \times \Y)$ be the kernel defining the equivalence, and let
$$F = F_\E: \Aut^0 (\X) \times \Pic^0 (\X) \longrightarrow \Aut^0 (\Y) \times \Pic^0 (\Y)$$ 
be the isomorphism of algebraic groups given by Theorem \ref{rouquier}. To prove the
assertion, we study the map 
$$\pi \colon \Pic^0( \X) \to \Aut^0 (\Y), \quad \pi(L) =  p_1 \bigl( F(\id, L) \bigr).$$ 
Note that if $ \psi = \pi(L)$, so that $F(\id, L) = (\psi, M)$, then by Theorem \ref{rouquier} we have
\begin{equation} \label{eq:formula-shE}
	p_1^{\ast} L \otimes \EE \simeq 
		p_2^{\ast} M \otimes (\id \times \psi)_{\ast} \EE. 
\end{equation}

The abelian variety $A = {\rm Im}~\pi$ naturally acts on $\Y$ by automorphisms. If $\dim A = 0$, 
i.e. $\pi (L) = \id$ for all $L \in \Pic^0 (\X)$, then $F$ restricted to $\{\id\} \times \Pic^0 (\X)$ induces 
an embedding of $\Pic^0(\X)$ into $\Pic^0(\Y)$, and therefore 
$h^{0,1}_{{\rm orb}} (\X)  \le h^{0,1}_{{\rm orb}} (\Y)$ by Lemma \ref{orbifold_picard}.

We can thus assume that $\dim A > 0$. We have a commutative diagram
$$\xymatrix{
\Pic^0 (\X)\ar[dr]_{\hspace{4mm} \pi}  \ar[r] & A \ar[r]^{ g \hspace{4mm}}  \ar[d]^{i_A} & \Alb (\Y) \\
& \Aut^0 (\Y) \ar[r]^{p \hspace{4mm}}  & \Alb \big(\Aut^0 (\Y) \big)\ar[u]}
$$
which defines the morphism of abelian varieties $g$, where $p$ is the projection map in Chevalley's theorem.
Given that the kernel of $p$ is affine, the composition $p \circ i_A$ has finite kernel. On the other hand, the right 
vertical map is given by (\ref{matsumura}), so again has finite kernel by the Fujiki-Matsumura-Nishi Theorem. 
Consequently $g: A \to \Alb(\Y)$ has finite kernel as well, hence the dual map $g^*: \Pic^0(\Y) \to \Pic^0(A)$ is surjective.

Now take a point $(x,y)$ in the support of $\EE$, and consider the orbit map
$$	f \colon A \to  \Y \simeq \{x\}\times \Y, \quad a \mapsto (x, \psi_a (y)), $$
where $\psi_a$ denotes the automorphism of $\Y$ corresponding to $a$.
Let $\FF = \L f^* (i_x \times \id_{\Y})^{\ast} \EE \in \D (A)$, where $i_x: \{x\} \rightarrow \X$ is the inclusion; it is nontrivial by
our choice of $(x,y)$.

For $a \in A$, let $t_a \in \Aut^0(A)$ denote translation by
$a$. The identity in \eqref{eq:formula-shE} implies that
\[
	t_a^{\ast} \FF \simeq f^* M \otimes \FF,
\]
whenever $L \in \Pic^0(\X)$ is such that $F(\id, L) = (\psi_a, M)$. 
(Note that this means that the cohomology sheaves of $\F$ are semihomogeneous vector bundles on $A$.)
In particular consider $L \in {\rm Ker}~\pi$, so that $a$ is the identity of $A$ and correspondingly
$$\FF \simeq  f^*M \otimes  \FF.$$
At least one of the cohomology sheaves of $\FF$ is nontrivial and therefore has positive rank $r$; 
by passing to determinants we have $f^*M^{\otimes r} \simeq \OO_A$. 
We conclude that
$$F( r\cdot {\rm Ker} ~\pi) \subseteq  {\rm Ker} ~ g^*.$$
Finally, since $F$ is an isomorphism, combined with Lemma \ref{birat_hodge}  and recalling that $g^*$ is 
surjective, this implies
$$h^{0,1}_{{\rm orb}} (\X) - \dim A = \dim({\rm Ker}~ \pi) \leq \dim({\rm Ker} ~g^*) = h^{0,1}_{{\rm orb}} (\Y) - h^{0,1} (A),$$
and therefore $h^{0,1}_{{\rm orb}} (\X)  \le h^{0,1}_{{\rm orb}} (\Y)$. 

This concludes the proof of the fact that $\Pic^0 (\X)$ and $\Pic^0 (\Y)$ have the same dimension. 
We now use this to show that they are in fact isogenous. Consider the construction symmetric to $A$ above, namely the abelian variety $B = {\rm Im}~ \nu$, with
$$\nu \colon \Pic^0(\Y) \to \Aut^0 (\X) , \quad \nu(M) =  p_1 \bigl( F^{-1}(\id, M) \bigr).$$ 
We first claim that ${\rm Ker }~ \pi \simeq {\rm Ker}~\nu$. Indeed, ${\rm Ker }~ \pi$ consists of those pairs 
$(\id, L)$ such that $F(\id, L) = (\id, M)$ for some $M \in \Pic^0 (\Y)$. Since the definition of ${\rm Ker }~ \nu$ 
is similar in the reverse direction, and $F$ is an isomorphism, the assertion follows. Since 
$h^{0,1}_{{\rm orb}} (\X)  = h^{0,1}_{{\rm orb}} (\Y)$, this gives in particular 
$$\dim A = \dim B \quad \text{and} \quad \dim {\rm Ker} \nu = \dim {\rm Ker} f^*.$$

Recall now that for any $L$ and $M$ such that $F(\id, L) = (\id, M)$ we have that $f^* M$ is a torsion line bundle 
in $\Pic^0 (A)$, and in fact for all such $M$ there exists a fixed $r>0$ such that $f^* M^{\otimes r} \simeq \OO_A$.
Hence via multiplication by $r$ on $\Pic^0(\Y)$, ${\rm Ker }~ \nu$ is mapped onto ${\rm Ker} ~f^*$, so  
$B$ is isogenous to $\Pic^0 (A)$, hence to $A$. Finally, since we have extensions
$$0\rightarrow {\rm Ker }~ \pi \rightarrow \Pic^0(\X) \rightarrow A\rightarrow 0 {\rm~~and~~}
0\rightarrow {\rm Ker }~ \nu \rightarrow \Pic^0(\Y) \rightarrow B\rightarrow 0,$$ 
it follows that $\Pic^0 (\X)$ and $\Pic^0 (\Y)$ are isogenous.
\end{proof}

\section{Further remarks and problems}

\noindent
{\bf A note on the McKay correspondence.}
As emphasized in the Introduction, the main result of this paper provides a proof of Corollary \ref{Gorenstein} based on general methods, in particular not relying on the existence of crepant resolutions and the McKay correspondence. 
It is important to note however that the McKay correspondence does provide another proof, where the methods 
apply only to the particular case of Gorenstein quotient singularities of dimension (up to) three.

Indeed, it is known to begin with that threefolds with Gorenstein quotient singularities admit crepant resolutions \cite{roan}. In fact, a canonical such 
resolution is provided by Nakamura's $G$-Hilbert scheme,  \cite{BKR} Theorem 1.2.
Denote by $X^\prime$ and $Y^\prime$ such particular crepant resolutions of $X$ and $Y$. 
Following \cite{BKR} in the global quotient case, one has  a derived version of the McKay correspondence, namely 
$\D (\X) \simeq \D(X^\prime)$ and $\D(\Y) \simeq \D(Y^\prime)$.\footnote{I thank Y. Kawamata for pointing out to me that this holds in 
the general case as well. Indeed, the proof is based on checking an equivalence criterion on a spanning class in $\D (X^\prime)$; see \cite{BKR} 
\S2.6. As these spanning classes are canonically provided by the $G$-Hilbert scheme structure, the local ones glue to give natural spanning classes
in the general case.} 
From the induced equivalence $\D(X^\prime) \simeq \D(Y^\prime)$ of derived categories of smooth projective threefolds, we obtain 
according to \cite{PS} that
$$h^{p,q} (X^\prime) =  h^{p,q} (Y^\prime) \,\,\,\,{\rm for~all~} p, q.$$
On the other hand, in the presence of a crepant resolution, one has
$$h^{p,q}_{\rm orb} (\X) =  h^{p,q} (X^\prime) \,\,\,\,{\rm for~all~} p, q,$$
and similarly for $Y$, as follows from the cohomological McKay-Ruan correspondence proved in \cite{LP} and 
\cite{yasuda1}.\footnote{In the latter, Yasuda also shows that $h^{p,q}_{\rm orb} (\X)$ are equal to Batyrev's stringy 
Hodge numbers of $X$, while in \cite{yasuda2} he extends the cohomological McKay-Ruan correspondence to an
appropriate statement in the non-Gorenstein case.}

\noindent
{\bf Invariance of other orbifold Hodge numbers.}
Due to constraints on degree shifting numbers and the size of the fixed loci, the following analogue of a well-known result in the smooth case works even in the non-Gorenstein case.

\begin{corollary}\label{high}
Let $X$ and $Y$ be normal projective varieties with quotient singularities, of dimension $n$, and let $\X$ and $\Y$ be the associated smooth Deligne-Mumford stacks. Assume that $\D(\X) \simeq \D(\Y)$. Then 
$$h^{n,0}_{\rm orb} (\X) =  h^{n,0}_{\rm orb} (\Y) \,\,\,\, {\rm and} \,\,\,\, h^{n-1,0}_{\rm orb} (\X) =  h^{n-1,0}_{\rm orb} (\Y).$$
\end{corollary}
\begin{proof}
We have seen in Theorem \ref{general}(i) that
$$\underset{p-q= i}{\sum} h^{p,q}_{\rm orb}(\X) = \underset{p-q= i}{\sum} h^{p,q}_{\rm orb} (\Y)$$
for all $- n\le i \le n$, while in the proof of Lemma \ref{birat_hodge} it was pointed out that when $p$ or $q$ are $0$, the 
only contribution to $h^{p,q}_{\rm orb}(\X)$ can come from the untwisted sector, when $a(\Y) = 0$. 
By taking $i =n$, this immediately gives the statement for $h^{n,0}_{{\rm orb}}$.

Take now $i = n-1$. In this case, the contributions on say the left hand side are of the form 
$$\underset{p-q= n-1}{\sum} h^{p,q}_{\rm orb}(\X) = h^{n,1}_{\rm orb} (\X) + h^{n-1,0}_{\rm orb} (\X) +  \sum_{p-q = n-1} \sum_{\Z} h^{p - a(\Z), q - a(\Z)} (Z),$$ 
where the rightmost sum is taken over components $\Z \subset \I \X$  such that 
$a (\Z) > 0 $, and $p - q = n-1$. But note that each $Z$ corresponds to a nontrivial fixed locus, hence by 
hypothesis we have $\dim Z \le n-2$. This means that there are in fact no nontrivial contributions coming from this sum, and 
since $h^{n,1}_{\rm orb} (\X) =  h^{n-1,0}_{\rm orb} (\X)$ by duality, we obtain the derived invariance of  $h^{n-1,0}_{\rm orb} (\X) ( = h^{n-1, 0} (X))$.
\end{proof}

\noindent
{\bf Further questions.}
Due to comparison theorems of Yasuda \cite{yasuda1}, \cite{yasuda2}, the orbifold Hodge numbers 
in the previous statements are the same as the \emph{stringy} Hodge numbers of $h^{p,q}_{{\rm st}}(X)$ introduced by Batyrev. (Here I am glossing over the fact that correctly one should rather speak of the stringy $E$-function.)
One can hope for a more general picture extending beyond the case of varieties with quotient singularities, in which case only the notion of stringy Hodge numbers continues to make sense. Given any $\QQ$-Gorenstein log-terminal variety, Kawamata \cite{kawamata4}  has proposed the existence of a unique category $\D(X)$ associated to $X$ which is natural from the point of view of the minimal model program, 
behaves like the derived category of a smooth projective variety, and satisfies
$$f^* {\rm Perf}(X) \subset \D(X ) \subset \D(\widetilde X)$$
for any resolution of singularities $f: \widetilde X \rightarrow X$. He showed the existence of this category in some special case, and argued that
in the case when $X$ has quotient singularities it is given by the orbifold derived category $\D (\X)$. The following question is unfortunately not completely well-posed at the moment, but worth asking in view of Kawamata's program.

\begin{question}
Let $X$ and $Y$ be $\QQ$-Gorenstein log-terminal projective varieties such that $\D(X) \simeq \D(Y)$. Does the equality of 
stringy Hodge numbers 
$$h^{p,q}_{\rm st} (X) =  h^{p,q}_{\rm st} (Y)$$
hold for each $p$ and $q$ whenever all the quantities in the statement are well-defined?
\end{question}

Another natural question, this time back in the realm of quotient singularities, stems from the fact that in the non-Gorenstein case 
the isomorphism in Theorem \ref{general}(i) involves 
orbifold Dolbeaut cohomology spaces $H^{p,q}_{{\rm orb}} (\X)$ with $p, q \not\in \ZZ$. 
I do not know at this stage whether there is a more refined way of matching these with special subspaces on the other side, 
thus providing more relations in the orbifold Hodge numbers than the naive columns of the Hodge diamond, but it would 
of course be extremely useful if such a construction could be found; for instance, as T. Yasuda suggests, it is possible that the 
Frobenius action on the $\ell$-adic version of orbifold cohomology as in \cite{rose} might help in this direction.

\end{document}

%% file: abbreviations-cd.tex
\usepackage{amsmath,amsfonts,amsthm,mathrsfs}
\usepackage{amssymb}

\usepackage[unicode,bookmarks]{hyperref}
\usepackage[usenames,dvipsnames]{xcolor}
\hypersetup{colorlinks=true,citecolor=NavyBlue,linkcolor=BrickRed,urlcolor=Orange}

\usepackage[alphabetic,initials]{amsrefs}

\usepackage{ifthen}

\usepackage{enumitem}

\usepackage{chngcntr}

\usepackage{tikz}
\usetikzlibrary{matrix,arrows}
\newlength{\myarrowsize} 
\pgfarrowsdeclare{cmto}{cmto}{
	\pgfsetdash{}{0pt} 
	\pgfsetbeveljoin 
	\pgfsetroundcap 
	\setlength{\myarrowsize}{0.6pt}
	\addtolength{\myarrowsize}{.5\pgflinewidth}
	\pgfarrowsleftextend{-4\myarrowsize-.5\pgflinewidth} 
	\pgfarrowsrightextend{.8\pgflinewidth}
}{
	\setlength{\myarrowsize}{0.6pt} 
  	\addtolength{\myarrowsize}{.5\pgflinewidth}  
	\pgfsetlinewidth{0.5\pgflinewidth}
	\pgfsetroundjoin
	\pgfpathmoveto{\pgfpoint{1.5\pgflinewidth}{0}}
	\pgfpatharc{-109}{-170}{4\myarrowsize}
	\pgfpatharc{10}{189}{0.58\pgflinewidth and 0.2\pgflinewidth}
	\pgfpatharc{-170}{-115}{4\myarrowsize+\pgflinewidth}
	\pgfpathclose
	\pgfusepathqfillstroke
	\pgfpathmoveto{\pgfpoint{1.5\pgflinewidth}{0}}
	\pgfpatharc{109}{170}{4\myarrowsize}
	\pgfpatharc{-10}{-189}{0.58\pgflinewidth and 0.2\pgflinewidth}
	\pgfpatharc{170}{115}{4\myarrowsize+\pgflinewidth}
	\pgfpathclose
	\pgfusepathqfillstroke
	\pgfsetlinewidth{2\pgflinewidth}
}

\pgfarrowsdeclare{cmonto}{cmonto}{
	\pgfsetdash{}{0pt} 
	\pgfsetbeveljoin 
	\pgfsetroundcap 
	\setlength{\myarrowsize}{0.6pt}
	\addtolength{\myarrowsize}{.5\pgflinewidth}
	\pgfarrowsleftextend{-4\myarrowsize-.5\pgflinewidth} 
	\pgfarrowsrightextend{.8\pgflinewidth}
}{
	\setlength{\myarrowsize}{0.6pt} 
  	\addtolength{\myarrowsize}{.5\pgflinewidth}  
	\pgfsetlinewidth{0.5\pgflinewidth}
	\pgfsetroundjoin
	\pgfpathmoveto{\pgfpoint{1.5\pgflinewidth}{0}}
	\pgfpatharc{-109}{-170}{4\myarrowsize}
	\pgfpatharc{10}{189}{0.58\pgflinewidth and 0.2\pgflinewidth}
	\pgfpatharc{-170}{-115}{4\myarrowsize+\pgflinewidth}
	\pgfpathclose
	\pgfusepathqfillstroke
	\pgfpathmoveto{\pgfpoint{1.5\pgflinewidth}{0}}
	\pgfpatharc{109}{170}{4\myarrowsize}
	\pgfpatharc{-10}{-189}{0.58\pgflinewidth and 0.2\pgflinewidth}
	\pgfpatharc{170}{115}{4\myarrowsize+\pgflinewidth}
	\pgfpathclose
	\pgfusepathqfillstroke
	\pgfpathmoveto{\pgfpoint{1.5\pgflinewidth-0.3em}{0}}
	\pgfpatharc{-109}{-170}{4\myarrowsize}
	\pgfpatharc{10}{189}{0.58\pgflinewidth and 0.2\pgflinewidth}
	\pgfpatharc{-170}{-115}{4\myarrowsize+\pgflinewidth}
	\pgfpathclose
	\pgfusepathqfillstroke
	\pgfpathmoveto{\pgfpoint{1.5\pgflinewidth-0.3em}{0}}
	\pgfpatharc{109}{170}{4\myarrowsize}
	\pgfpatharc{-10}{-189}{0.58\pgflinewidth and 0.2\pgflinewidth}
	\pgfpatharc{170}{115}{4\myarrowsize+\pgflinewidth}
	\pgfpathclose
	\pgfusepathqfillstroke
	\pgfsetlinewidth{2\pgflinewidth}
}

\pgfarrowsdeclare{cmhook}{cmhook}{
	\pgfsetdash{}{0pt} 
	\pgfsetbeveljoin 
	\pgfsetroundcap 
	\setlength{\myarrowsize}{0.6pt}
	\addtolength{\myarrowsize}{.5\pgflinewidth}
	\pgfarrowsleftextend{-4\myarrowsize-.5\pgflinewidth} 
	\pgfarrowsrightextend{.8\pgflinewidth}
}{
	\setlength{\myarrowsize}{0.6pt} 
  	\addtolength{\myarrowsize}{.5\pgflinewidth}  
 	\pgfsetdash{}{0pt}
	\pgfsetroundcap
	\pgfpathmoveto{\pgfqpoint{0pt}{-4.667\pgflinewidth}}
	\pgfpathcurveto
    {\pgfqpoint{4\pgflinewidth}{-4.667\pgflinewidth}}
    {\pgfqpoint{4\pgflinewidth}{0pt}}
    {\pgfpointorigin}
	\pgfusepathqstroke
}

\newenvironment{diagram}[2]{%
\begin{equation}%
\begin{tikzpicture}[>=cmto,baseline=(current bounding box.center),%
	to/.style={-cmto,font=\scriptsize,cap=round},%
	into/.style={cmhook->,font=\scriptsize,cap=round},%
	onto/.style={-cmonto,font=\scriptsize,cap=round},%
	math/.style={matrix of math nodes, row sep=#2, column sep=#1,%
		text height=1.5ex, text depth=0.25ex}]%
}{%
\end{tikzpicture}%
\end{equation}% 
\ignorespacesafterend%
}
\newenvironment{diagram*}[2]{%
\[%
\begin{tikzpicture}[>=cmto,baseline=(current bounding box.center),%
	to/.style={->,font=\scriptsize,cap=round},%
	into/.style={cmhook->,font=\scriptsize,cap=round},%
	onto/.style={-cmonto,font=\scriptsize,cap=round},%
	math/.style={matrix of math nodes, row sep=#2, column sep=#1,%
		text height=1.5ex, text depth=0.25ex}]%
}{%
\end{tikzpicture}%
\]% 
\ignorespacesafterend%
}

\newcommand{\decal}[1]{\lbrack #1 \rbrack}

\newcommand{\ltriangle}[4][]%
{\begin{diagram}[#1]%
	{#2} &\rTo& {#3} &\rTo& {#4} &\rTo& {#2 \decal{1}}%
\end{diagram}}

\newcommand{\ZZ}{\mathbb{Z}}
\newcommand{\QQ}{\mathbb{Q}}
\newcommand{\RR}{\mathbb{R}}
\newcommand{\CC}{\mathbb{C}}

\DeclareMathOperator{\id}{id}

\def\overbar#1#2#3{{%
	\setbox0=\hbox{\displaystyle{#1}}%
	\dimen0=\wd0
	\advance\dimen0 by -#2 
	\vbox {\nointerlineskip \moveright #3 \vbox{\hrule height 0.3pt width \dimen0}%
		\nointerlineskip \vskip 1.5pt \box0}%
}}

\makeatletter
\let\@@seccntformat\@seccntformat
\renewcommand*{\@seccntformat}[1]{%
  \expandafter\ifx\csname @seccntformat@#1\endcsname\relax
    \expandafter\@@seccntformat
  \else
    \expandafter
      \csname @seccntformat@#1\expandafter\endcsname
  \fi
    {#1}%
}
\newcommand*{\@seccntformat@subsection}[1]{%
  \textbf{\csname the#1\endcsname.}
}
\makeatother

\makeatletter
\let\@paragraph\paragraph
\renewcommand*{\paragraph}[1]{%
	\vspace{0.3\baselineskip}%
	\@paragraph{\textit{#1}}%
}
\makeatother

\counterwithin{equation}{subsection}
\counterwithout{subsection}{section}
\counterwithin{figure}{subsection}

\newtheorem{theorem}[equation]{Theorem}
\newtheorem*{theorem*}{Theorem}
\newtheorem{lemma}[equation]{Lemma}
\newtheorem*{lemma*}{Lemma}
\newtheorem{corollary}[equation]{Corollary}
\newtheorem{proposition}[equation]{Proposition}
\newtheorem*{proposition*}{Proposition}

\theoremstyle{definition}

\newtheorem*{definition*}{Definition}
\theoremstyle{remark}
\newtheorem*{remark}{Remark}

\newtheorem*{question}{Question}

\newtheorem*{example*}{Example}

\theoremstyle{plain}

\newcommand{\theoremref}[1]{\hyperref[#1]{Theorem~\ref*{#1}}}
\newcommand{\lemmaref}[1]{\hyperref[#1]{Lemma~\ref*{#1}}}
\newcommand{\propositionref}[1]{\hyperref[#1]{Proposition~\ref*{#1}}}
\newcommand{\conjectureref}[1]{\hyperref[#1]{Conjecture~\ref*{#1}}}
\newcommand{\corollaryref}[1]{\hyperref[#1]{Corollary~\ref*{#1}}}

\makeatletter
\let\old@caption\caption
\renewcommand*{\caption}[1]{%
	\setcounter{figure}{\value{equation}}%
	\stepcounter{equation}%
	\old@caption{#1}\relax%
}
\makeatother

\newcounter{thmA}
\newtheorem{theorem-intro}[thmA]{Theorem}